\documentclass[11pt]{amsart}
\usepackage{amsmath}
\usepackage{amsfonts}
\usepackage{amsthm}
\usepackage{amssymb}
\usepackage{times}

\usepackage{url}

\usepackage{ifthen}
\usepackage{ifpdf}

\usepackage[pdfauthor={Achill Schuermann},
            pdftitle={Perfect, strongly eutactic lattices are periodic extreme},
            pdfsubject={showing that perfect, strongly eutactic lattices cannot locally
                        be modified to yield a better periodic sphere packing}]{hyperref}

\renewcommand{\vec}[1]{\boldsymbol{#1}}

\newcommand{\N}{\mathbb N}
\newcommand{\Z}{\mathbb Z}

\newcommand{\R}{\mathbb R}

\newcommand{\MF}{{\mathcal F}}
\newcommand{\MH}{{\mathcal H}}

\newcommand{\MNR}{{\mathcal R}}
\newcommand{\MS}{{\mathcal S}}
\newcommand{\MP}{{\mathcal P}}

\newcommand{\MV}{{\mathcal V}}

\newcommand{\MU}{{\mathcal U}}

\newcommand{\GL}{\mathsf{GL}}

\newcommand{\gldz}{\GL_d(\Z)}

\newcommand{\gldr}{\GL_d(\R)}

\newcommand{\sd}{{\mathcal S}^d}
\newcommand{\sdm}{{\mathcal S}^{d,m}}
\newcommand{\sdo}{{\mathcal S}^d_{>0}}
\newcommand{\sdmo}{{\mathcal S}^{d,m}_{>0}}

\DeclareMathOperator{\trace}{trace}

\DeclareMathOperator{\vol}{vol}

\DeclareMathOperator{\cone}{cone}

\DeclareMathOperator{\Min}{Min}

\DeclareMathOperator{\card}{card}
\DeclareMathOperator{\bd}{bd}
\DeclareMathOperator{\relint}{relint}
\DeclareMathOperator{\interior}{int}

\DeclareMathOperator{\grad}{grad}
\DeclareMathOperator{\hess}{hess}

\newtheorem{definition}{Definition}
\newtheorem{proposition}[definition]{Proposition}
\newtheorem{lemma}[definition]{Lemma}
\newtheorem{theorem}[definition]{Theorem}
\newtheorem{corollary}[definition]{Corollary}

\theoremstyle{definition}
\newtheorem*{example}{Example}
\newtheorem{remark}[definition]{Remark}

\newcounter{tab}

\newenvironment{bigtab}{\medskip%
                        \refstepcounter{tab}%
                        \begin{table}[htbp]%
                        \begin{center}}%
                       {\end{center}\end{table}\medskip}

\author{Achill Sch\"urmann}
\address{Institute of Applied Mathematics, Delft University of Technology, Mekelweg~4, 2628~CD Delft, The Netherlands}
\email{a.schurmann@tudelft.nl}
\thanks{The author was supported by the 
Deutsche Forschungsgemeinschaft (DFG) under grant SCHU 1503/4-2. 
He thanks the Hausdorff Research Institute for Mathematics for its 
hospitality and support.}

\title[Perfect, strongly eutactic lattices are periodic extreme]{Perfect, strongly eutactic lattices\\ are periodic extreme}

\begin{document}

\sloppy


\begin{abstract}
We introduce a parameter space for 
periodic point sets, given as unions of $m$~translates of point lattices.
In it we investigate 
the behavior of the sphere packing density 
function and derive sufficient conditions for local optimality.
Using these criteria we prove that perfect, 
strongly eutactic lattices cannot be locally 
improved to yield a periodic sphere packing 
with greater density. This applies in particular 
to the densest known lattice sphere packings 
in dimension $d\leq 8$ and $d=24$.
\end{abstract}


\subjclass[2000]{52C17, 11H55}


\maketitle

\centerline{\em Dedicated to Jacques Martinet on the occasion of his 70th birthday.}

\bigskip
\bigskip


%
%
%

\section{Introduction}

The classical and widely studied {\em sphere packing problem}
asks for a non-overlapping arrangement of equally sized 
spheres in a Euclidean space, 
such that the fraction of space covered by spheres is maximized.
The problem arose from the arithmetical study of positive 
definite quadratic forms.
By the works Thue \cite{thue-1910} and Hales \cite{hales-2005} 
the optimal arrangements of spheres are known in dimension~$2$ and~$3$.
We refer to \cite{gl-1987}, \cite{cs-1998}, \cite{martinet-2003}
and \cite{schuermann-2009} for details and further reading.

For reasons related to the historical roots of the sphere packing
problem, special attention has been on {\em (point) lattices}
as the discrete set of sphere centers.
In dimension~$2$ the {\em hexagonal lattice} and in dimension~$3$
the {\em face-centered-cubic lattice}
yield optimal sphere packings. 
For the restriction of the sphere packing problem to lattices,
the optimal configurations are known up to 
dimension~$8$ and in dimension~$24$ (see Table~\ref{tab:sphere-packing-results}).
Here, solutions are given by fascinating objects, 
the so-called {\em root lattices} and the {\em Leech lattice}.
We refer to \cite{cs-1998}, \cite{martinet-2003}
and \cite{ns-2008} for further information on these exceptional objects.

A major {\bf open problem} in the theory of sphere packing is to find a 
dimension in which there is a non-lattice packing that is denser than
any lattice packing.
In dimension~$10$ there exists a non-lattice sphere packing,
that is conjectured to have a higher density than any 
lattice sphere packing (see \cite{ls-1970}).
As shown in Table~\ref{tab:sphere-packing-results}, 
below dimension~$24$ similar sphere packings have been found 
in dimensions~$11$, $13$, $18$, $20$ and $22$.
All of them are {\em periodic}, that is,
a finite union of translates of a lattice sphere packing.
By a well-known conjecture, attributed by Gruber \cite{gruber-2007}
to Zassenhaus, optimal sphere packing density can always be attained 
by periodic sphere packings. It is known that 
their density comes arbitrarily close to the optimal value
(see for example~\cite[Appendix A]{ce-2003}).

A natural idea to obtain a better non-lattice sphere packing,
is to ``locally modify'' one of the optimal known lattice sphere packings 
in dimensions $d=4,\ldots,8$.
In this paper we show that such modifications are not possible 
within the set of all periodic sphere packings
(see Corollary \ref{cor:root_lattice_periodic_extreme}).
We more generally show in Theorem \ref{thm:main-periodic}
that such modifications are not possible for 
{\em perfect, strongly eutactic lattices}.

One may wonder why the restriction to periodic structures is
necessary. One could also consider more general discrete sets.
However, within the set of all discrete sets, we are not aware of
any notion of a ``local modification'' that on the one hand could
potentially lead to an improved sphere packing density, but on
the other hand would allow us to generalize the result of this paper.
For instance, a natural approach to define the $\epsilon$-neighborhood 
of a discrete set is as the collection of sets that can be obtained 
by changing the position of elements by at most an $\epsilon$~distance. 
However, such a local modification would not even 
change the sphere packing density. It is equal to a constant multiple of
the average number of points per unit volume, which could not be changed 
in such an $\epsilon$-neighborhood. In contrast to that, the local changes of 
periodic sets considered in this paper allow arbitrarily large displacements of points,
if they are far enough from the origin.

The paper is organized as follows.
In Section~\ref{sec:background} we recall some necessary background on
lattices and positive definite quadratic forms. 
In Section~\ref{sec:ryshkov} we introduce the so-called Ryshkov polyhedron,
and based on it we give a geometrical interpretation of Voronoi's characterization 
of locally optimal lattice sphere packings.
This viewpoint allows a natural generalization to study
local optimal periodic sphere packings. 
For their study we introduce a parameter space 
in Section~\ref{sec:periodic-parameter-space}.
We give characterizations of local optimal periodic sphere packings
with up to $m$~lattices translates in Section~\ref{sec:local-analysis}.
Based on these general characterizations we obtain one of the main results of
this paper in
Section~\ref{sec:periodic-extreme}: We show that
perfect, strongly eutactic lattices cannot locally be modified to yield 
a better periodic sphere packing -- they are  {\em periodic extreme} 
(see Definition~\ref{def:periodic-extreme}).

\begin{bigtab} \label{tab:sphere-packing-results}
\begin{tabular}{c|c|c|c}
$d$ & point set & $\quad\delta/\vol B^d\quad$ & author(s) \\
\hline
 $2$ & ${\mathsf A}_2$    &  $0.2886\ldots$  &  Lagrange, 1773, \cite{lagrange-1773}\\
 $3$ & ${\mathsf A}_3={\mathsf D}_3, \mathsf{\ast}$    &  $0.1767\ldots$  &  Gau{\ss}, 1840, \cite{gauss-1840}\\
 $4$ & ${\mathsf D}_4$    &  $0.125\phantom{0\ldots}$  &  Korkine \& Zolotareff, 1877, \cite{kz-1877}\\
 $5$ & ${\mathsf D}_5,\mathsf{\ast}$    &  $0.0883\ldots$  &  Korkine \& Zolotareff, 1877, \cite{kz-1877}\\
 $6$ & ${\mathsf E}_6,\mathsf{\ast}$    &  $0.0721\ldots$  &  Blichfeldt, 1935, \cite{blichfeldt-1934}\\
 $7$ & ${\mathsf E}_7,\mathsf{\ast}$    &  $0.0625\phantom{\ldots}$  &  Blichfeldt, 1935, \cite{blichfeldt-1934}\\
 $8$ & ${\mathsf E}_8$    &  $0.0625\phantom{\ldots}$  &  Blichfeldt, 1935, \cite{blichfeldt-1934}\\
 $9$ & ${\mathsf \Lambda}_9,\mathsf{\ast}$    &  $0.0441\ldots$  &  \\
 $10$ & ${\mathsf P}_{10c}$    &  $0.0390\ldots$  & Leech \& Sloane, 1970, \cite{ls-1970}\\ 
 $11$ & ${\mathsf P}_{11a}$    &  $0.0351\ldots$  & Leech \& Sloane, 1970, \cite{ls-1970}\\ 
 $12$ & ${\mathsf K}_{12}$    &  $0.0370\ldots$  &  \\ 
 $13$ & ${\mathsf P}_{13a}$    &  $0.0351\ldots$  & Leech \& Sloane, 1970, \cite{ls-1970}\\ 
 $14$ & ${\mathsf \Lambda}_{14},\mathsf{\ast}$    &  $0.0360\ldots$  &  \\
 $15$ & ${\mathsf \Lambda}_{15},\mathsf{\ast}$    &  $0.0441\ldots$  &  \\
 $16$ & ${\mathsf \Lambda}_{16},\mathsf{\ast}$    &  $0.0625\phantom{\ldots}$  &  \\
 $17$ & ${\mathsf \Lambda}_{17},\mathsf{\ast}$    &  $0.0625\phantom{\ldots}$  &  \\
 $18$ & ${\mathsf V}_{18}$    &  $0.0750\ldots$  &  Bierbrauer \& Edel, 1998, \cite{be-2000}\\
 $19$ & ${\mathsf \Lambda}_{19},\mathsf{\ast}$    &  $0.0883\ldots$  &  \\
 $20$ & ${\mathsf V}_{20}$    &  $0.1315\ldots$  &  Vardy, 1995, \cite{vardy-1995}\\
 $21$ & ${\mathsf \Lambda}_{21},\mathsf{\ast}$    &  $0.1767\ldots$  &  \\
 $22$ & ${\mathsf V}_{22}$    &  $0.3325\ldots$  &  Conway \& Sloane, 1996, \cite{cs-1996}\\
 $23$ & ${\mathsf \Lambda}_{23}$    &  $0.5\phantom{000\ldots}$  &  \\
 $24$ & ${\mathsf \Lambda}_{24}$    &  $1\phantom{.0000\ldots}$  &  Cohn \& Kumar, 2004, \cite{ck-2004}\\
\end{tabular}
\par\medskip
\textbf{Table \arabic{tab}. } 
Point sets defining best known sphere packings up to dimension $24$.
In dimensions $d\leq 8$ and $d=24$ the corresponding authors solved the lattice sphere packing problem.
The other mentioned authors found the listed, densest known periodic sphere packings.
The asterisk~$\mathsf{\ast}$ indicates that an equally dense, periodic non-lattice sphere packing is known.
\end{bigtab}

\section{Background on lattices and quadratic forms}

\label{sec:background}

\subsection*{Lattices and Periodic Sets}

A (full rank) {\em lattice} $L$ in $\R^d$ is a discrete subgroup
$L=\Z \vec{a}_1 + \ldots + \Z \vec{a}_d$ generated by $d$ linear independent 
(column) vectors $\vec{a}_i\in\R^d$. We say that these vectors 
form a {\em basis} of $L$ and associate it with the matrix
$A=(\vec{a}_1,\ldots, \vec{a}_d)\in\gldr$. 
We write $L=A\Z^d$. It is well-known that $L$ is generated in this way
precisely by the matrices $AU$ with $U\in\gldz$.
We refer to \cite{gl-1987} for details and more background on lattices.
Given a lattice $L$ and {\em translational vectors} $\vec{t}_i$, 
for say $i=1,\ldots,m$, the discrete set 
\begin{equation}  \label{eqn:periodic-set}
\Lambda = \bigcup_{i=1}^m \left( \vec{t}_i + L \right)
\end{equation}
is called a {\em periodic (point) set}.

The {\em sphere packing radius} $\lambda(\Lambda)$ of a discrete set~$\Lambda$
(not necessarily periodic) 
in the Euclidean space $\R^d$ (with norm $\|\cdot \|$)
is defined as the infimum of half the distances between distinct points: 
\begin{equation*} 
\lambda(\Lambda) = \frac{1}{2} \inf_{\vec{x},\vec{y}\in\Lambda, \vec{x}\not=\vec{y}} \|\vec{x}-\vec{y}\|.
\end{equation*}
The sphere packing radius is the largest possible radius $\lambda$ such that
solid spheres of radius~$\lambda$ and with centers in~$\Lambda$
do not overlap.
Denoting the solid unit sphere by $B^d$, the {\em sphere packing} defined
by $\Lambda$ is the union of non-overlapping spheres
$$
\bigcup_{\vec{x}\in \Lambda} \left( \vec{x} + \lambda(\Lambda) B^d \right).
$$
Its {\em density} $\delta(\Lambda)$ is, loosely speaking, defined as the fraction of space covered by spheres. 
We can make this definition more precise 
by considering a cube 
$$
C=\{\vec{x}\in \R^d : |x_i| \leq 1/2 \}
$$ 
and setting
$$
\delta(\Lambda)
=
\lambda(\Lambda)^d \vol B^d \cdot
\liminf_{\lambda\to\infty} \frac{\card (\Lambda \cap \lambda C)}{\vol \lambda C}
.
$$

If the limit inferior above is a true limit, the cube in the definition 
can be replaced by any other compact set~$C$ that is the closure of its interior,
without the value of $\delta$ changing.
We say that a corresponding set $\Lambda$ is {\em uniformly dense} in that case.
It can be shown that the supremum of $\delta(\Lambda)$ over all discrete sets
is attained by a uniformly dense set $\Lambda$.
We refer to \cite{groemer-1963} and \cite[Appendix A]{ce-2003} for further reading.

For general discrete sets, it may be difficult to compute the density,
respectively the limit inferior in the definition. For a lattice 
the limit inferior can simply be replaced by $1/\det L$,
where $\det L = |\det A|$ is the {\em determinant} of the lattice~$L=A\Z^d$.
Note that the determinant of $L$ 
is independent of the particular choice of the basis~$A$.
For periodic sets $\Lambda$ as in \eqref{eqn:periodic-set}
we get the estimate 
$$
\delta(\Lambda) \leq \frac{m \lambda(\Lambda)^d \vol B^d}{\det L} , 
$$
with equality if and only if the lattice translates 
$\vec{t}_i+L$ are pairwise disjoint. 

%
%
%

\subsection*{Positive definite quadratic forms}

\label{sec:background-pqf}

Among similarity classes of lattices, hence in the space
$O_d(\R)\backslash \gldr / \gldz$, there exist only finitely many
local maxima of~$\delta$ up to scaling. In order to characterize and to work with them,
i.e., enumerate them, it is convenient to use the language of
real {\em positive definite quadratic forms} (PQFs for short). 
These are simply identified with 
the set $\sdo$ of real symmetric, positive definite matrices. 
Given a matrix $Q\in\sdo$, 
we set $Q[\vec{x}]=\vec{x}^tQ\vec{x}$ for $\vec{x}\in\R^d$, 
defining a corresponding PQF.
Note that every matrix $Q\in\sdo$ can be decomposed into $Q=A^tA$
with $A\in\gldr$ and therefore $\sdo$ can be identified with
the space $O_d(\R)\backslash \gldr$ of lattice bases up to orthogonal transformations.
Two PQFs (respectively matrices) $Q$ and $Q'$ are called {\em arithmetically equivalent}
(or {\em integrally equivalent})
if there exists a matrix $U\in \gldz$ with $Q'=U^tQU$.
Thus arithmetical equivalence classes of PQFs are in one-to-one correspondence
with similarity classes of lattices.

The {\em arithmetical minimum} $\lambda(Q)$ of a PQF $Q$ is defined by
$$
\lambda(Q) 
=
\min_{\vec{x}\in \Z^d\setminus\{\vec{0}\}} Q[\vec{x}]
.
$$
If $L=A\Z^d$ with $A\in\gldr$ satisfying $Q=A^tA$ is a corresponding lattice,
there is an immediate relation to the packing radius of $L$:
We have $\lambda(Q)=(2\lambda(L))^2$ and therefore
$$
\delta(L)= \MH(Q)^{d/2} \frac{\vol B^d}{2^d},
$$
where 
$$
\MH(Q) = \frac{\lambda(Q)}{(\det Q)^{1/d}}
$$
is the so-called {\em Hermite invariant} of~$Q$.
Note that $\MH(\cdot)$ is invariant with respect to scaling.
A classical problem in the arithmetic theory of quadratic forms is the 
determination of the {\em Hermite constant} 
$$
\MH_d=\sup_{Q\in\sdo} \MH(Q)
.
$$
By the relation described above, it corresponds to determining the
supremum of possible lattice sphere packing densities.
Local maxima of the Hermite invariant on $\sdo$ and corresponding lattices
are called {\em extreme}.

\section{Voronoi's characterization of extreme forms}

\label{sec:ryshkov}

\subsection*{The Ryshkov polyhedron}

Since the Hermite invariant is invariant with respect to scaling,
a natural approach to maximizing it
is to consider all forms with a fixed 
arithmetical minimum, say~$1$, and minimize the determinant
among them.
We may even relax the condition on the arithmetical minimum
and only require that it is at least~$1$.
In other words, we have
$$
\MH_d
=
1 / \inf_{\MNR} (\det Q)^{1/d}
,
$$
where  
\begin{equation}  \label{eqn:ryshkov-polyhedron}
\MNR =
\left\{
Q\in\sdo : \lambda(Q)\geq 1
\right\}
.
\end{equation}
We refer to $\MNR$ as {\em Ryshkov polyhedron},
as it was Ryshkov \cite{ryshkov-1970}
who noticed that this view
on Hermite's constant allows a simplified 
description of Voronoi's theory, to be sketched below.

We denote by $\sd$ the space of real symmetric matrices,
respectively of real quadratic forms in $d$~variables.
It is a Euclidean vector space of dimension $\binom{d+1}{2}$
with the usual inner product defined by
$$
\langle Q, Q' \rangle
=
\sum_{i,j=1}^d q_{ij} q'_{ij}
=
\trace(Q\cdot Q')
.
$$
Because of the fundamental identity
\[
Q[\vec{x}] = \langle Q,\vec{x}\vec{x}^t \rangle 
,
\]
quadratic forms $Q\in \sd$ attaining a fixed value 
on a given $\vec{x}\in\R^d\setminus\{\vec{0}\}$ 
lie all in a {\em hyperplane}
({\em affine subspace of co-dimension~$1$}).
Thus Ryshkov polyhedra $\MNR$ are intersections of 
infinitely many {\em halfspaces}:   
\begin{equation}  \label{eqn:Plambda}
\MNR = 
\{
Q\in\sdo : \langle Q, \vec{x}\vec{x}^t \rangle \geq 1 
\mbox{ for all } \vec{x}\in\Z^d\setminus\{\vec{0}\}
\}
.
\end{equation}

It can be shown that $\MNR$ is ``locally like a polyhedron'',
meaning that any intersection with a {\em polytope} 
(convex hull of finitely many vertices) is itself a polytope. 
For a proof we refer to \cite[Theorem~3.1]{schuermann-2009}.  
As a consequence $\MNR$ has {\em vertices}, {\em edges}, {\em facets} 
and in general {\em $k$-dimensional faces} ({\em $k$-faces}). 
For details on terminology and basic properties of 
polytopes we refer to \cite{ziegler-1998}.

\subsection*{Perfect forms}

The vertices $Q$ of the Ryshkov polyhedron are called {\em perfect forms}. 
Such forms are characterized by the fact that they are 
determined uniquely by their arithmetical minimum (here $1$)
and its representatives 
$$    
\Min Q =\{\vec{x}\in \Z^d : Q[\vec{x}] = \lambda(Q) \}
.
$$
A corresponding lattice is called perfect too.
The following proposition due to Minkowski implies that 
the Hermite constant can only be attained among perfect forms,
i.e., the maximal lattice sphere packing density 
can only be attained by perfect lattices.

\begin{proposition}[Minkowski~\cite{minkowski-1905}]  \label{prop:concave-det}
$(\det Q)^{1/d}$ is a strictly concave function on~$\sdo$.
\end{proposition}

For a proof see for example \cite[\textsection~39.2]{gl-1987}.
Note, that in contrast to~$(\det Q)^{1/d}$, the function $\det Q$ is not 
a concave function on~$\sdo$ (see \cite{nelson-1974}).
However Minkowski's theorem implies that the set
\begin{equation} \label{eqn:det-greater-equal-D}
\{Q\in \sdo : \det Q \geq D \}
\end{equation}
is strictly convex for $D>0$.

Another property of perfect forms which we use later is the following.

\begin{proposition}  \label{prop:prefect-implies-d-linear-indep}
If $Q\in\sd$ is perfect, then $\Min Q$ spans $\R^d$.
\end{proposition}

The existence of $d$ linear independent vectors in $\Min Q$ for a perfect
form $Q$ follows from the observation that the rank-$1$ forms
$\vec{x}\vec{x}^t$ with $\vec{x}\in\Min Q$ have to span $\sd$,
since they uniquely determine $Q$ through the linear 
equations $\langle Q, \vec{x}\vec{x}^t \rangle = \lambda(Q)$.
If however $\Min Q$ does not span $\R^d$ then these rank-$1$
forms can maximally span a $\binom{d}{2}$-dimensional subspace of $\sd$.

\subsection*{Finiteness up to equivalence}

The arithmetical equivalence operation $Q\mapsto U^t Q U$ of $\gldz$ on $\sdo$ 
leaves $\lambda(Q)$, $\Min Q$ and also $\MNR$ invariant.
In fact, $\gldz$ acts on the sets of faces of a given dimension,
thus in particular on the sets of vertices, edges and facets of 
$\MNR$.
The following theorem shows that 
the Ryshkov polyhedron $\MNR$ contains only finitely 
many arithmetically inequivalent vertices.
By Proposition~\ref{prop:concave-det} this implies in particular 
that $\MH_d$ is actually attained, namely by some perfect forms. 
For a proof we refer to \cite[Theorem~3.4]{schuermann-2009}.

\begin{theorem}[Voronoi \cite{voronoi-1907}] \label{thm:voronoi}
Up to arithmetical equivalence and scaling there exist
only finitely many perfect forms in a given dimension~$d\geq 1$.
\end{theorem}

Thus the classification of perfect forms in a given dimension, 
respectively the enumeration of vertices of the Ryshkov polyhedron 
up to arithmetical equivalence, yields the Hermite constant.
Perfect forms have been classified up to dimension~$8$
(see \cite{dsv-2007b}).

\subsection*{Characterization of extreme forms}

From dimension~$6$ onwards
not every perfect form is extreme (see \cite{martinet-2003}). 
In order to characterize extreme forms within the set of perfect forms
the notion of {\em eutaxy} is used:
A PQF $Q$ is called {\em eutactic}   
if its inverse~$Q^{-1}$ is contained in the (relative) interior 
$\relint \MV(Q)$ of its {\em Voronoi domain}
$$
\MV(Q)
= \cone \{ \vec{x}\vec{x}^t : \vec{x}\in\Min Q \}
.
$$
Here $\cone M $ denotes the {\em conic hull}
$$
\left\{
\sum_{i=1}^n \alpha_i \vec{x}_i : n\in \N \mbox{ and } \vec{x}_i \in M, \alpha_i\geq 0 \mbox{ for } i=1,\ldots, n
\right\}
$$
of a set $M$.
Note that the Voronoi domain is full-dimensional (of dimension $\binom{d+1}{2}$)
if and only if $Q$ is perfect.
Note also that the rank-$1$ forms $\vec{x}\vec{x}^t$ give
inequalities $\langle Q, \vec{x}\vec{x}^t \rangle \geq 1$ 
defining the Ryshkov polyhedron and by this the 
Voronoi domain of $Q$ is equal to the {\em normal cone}
\begin{equation} \label{eqn:normal-cone}
\{N \in \sd : \langle N, Q/\lambda(Q) \rangle \leq \langle N , Q' \rangle \mbox{ for all } Q'\in \MNR \}
\end{equation}
of $\MNR$ at its boundary point $Q/\lambda(Q)$.

Algebraically the eutaxy condition $Q^{-1} \in \relint \MV(Q)$
is equivalent to the existence
of positive $\alpha_{\vec{x}}$ with 
\begin{equation} \label{eqn:eutaxy-algebraic}
Q^{-1} = \sum_{\vec{x}\in\Min Q} \alpha_{\vec{x}} \vec{x}\vec{x}^t
.
\end{equation} 
Thus computationally eutaxy of~$Q$ can be tested by solving the {\em linear program}  
\begin{equation} \label{eqn:eutaxy-lp}
\max \alpha_{\min}
\quad \mbox{such that $\alpha_{\vec{x}}\geq \alpha_{\min}$ and \eqref{eqn:eutaxy-algebraic} holds.}
\end{equation}
The form $Q$ is eutactic if and only if 
the maximum is greater $0$.

Voronoi \cite{voronoi-1907} showed that perfection together with
eutaxy implies extremality and vice versa:

\begin{theorem}[Voronoi \cite{voronoi-1907}]   \label{thm:voronoi-packing}
A PQF $Q\in\sdo$ is extreme if and only if $Q$ is perfect
and eutactic.
\end{theorem}

We here give a proof providing a geometrical
viewpoint that turns out to be quite useful for the 
intended generalization discussed in the following sections.

\begin{proof}
The function~$\det Q$ is a positive real valued polynomial on $\sd$,
depending on the $\binom{d+1}{2}$ different coefficients $q_{ij}$ of $Q$. 
Using the expansion theorem we obtain 
$$
\det Q = \sum_{i=1}^d q_{ji}^\# q_{ij}
$$
for any fixed column index $j\in\{ 1,\dots,d \}$.
Here, $q_{ij}^\# = (-1)^{i+j} \det Q_{ij}$ 
(with $Q_{ij}$ the minor matrix of $Q$, obtained by removing row~$i$ and column~$j$)
denote the coefficients of
the {\em adjoint form} 
$Q^\#=(\det Q)Q^{-1}\in\sdo$ of $Q$.
Thus 
\begin{equation}  \label{eqn:gradient-det}
\grad \det Q = (\det Q)Q^{-1}
\end{equation}
and the tangent hyperplane $T$ in $Q$
of the smooth {\em determinant-$\det Q$-surface} 
$$
S = 
\{
Q'\in\sdo : \det Q' = \det Q
\}
$$
is given by
$$
T = 
\{
Q'\in\sd :
\langle Q^{-1},Q'\rangle
=
\langle Q^{-1},Q \rangle
\}
.
$$

Or in other words,
$Q^{-1}$ is a normal vector of 
the tangent plane $T$ of $S$ at $Q$. 
By Proposition~\ref{prop:concave-det}
and the observation that~\eqref{eqn:det-greater-equal-D} is convex,
we know that $S$ is contained in the halfspace
\begin{equation} \label{eqn:tangent-plane}
\{Q'\in\sd : \langle Q^{-1} , Q'-Q \rangle \geq 0 \}
,
\end{equation}
with $Q$ being the unique intersection point of $S$ and $T$.

As a consequence,
a perfect form $Q$ attains a local minimum of $\det Q$
(hence is extreme) if and only if the halfspace \eqref{eqn:tangent-plane}
contains the Ryshkov polyhedron $\MNR$, and its boundary meets $\MNR$
only in $Q$. This is easily seen to be equivalent to the condition that
the normal cone (Voronoi domain) $\MV(Q)$ of $\MNR$ at $Q$
contains $Q^{-1}$ in its interior.
\end{proof}

Note that eutaxy alone does not suffice for extremality.
However, there exist only finitely many eutactic forms in every 
dimension and they can (in principle) be enumerated too (see \cite[Section~9.5]{martinet-2003}).
Nevertheless, this seems computationally more difficult
than the enumeration of perfect forms
(see \cite{stogrin-1974}, \cite{bm-1996}, \cite{batut-2001}, \cite{egs-2002}).
By the geometry of~$S$ and~$T$ a eutactic form attains always a unique
minimum of~$\delta$ (maximum of $\det$) on its face of the Ryshkov polyhedron.
However, not all faces of the Ryshkov polyhedron contain a eutactic form.

\section{Parameter spaces for periodic sets}

\label{sec:periodic-parameter-space}

We want to study the more general situation of 
periodic sphere packings.
Recall from \eqref{eqn:periodic-set} that a periodic set
with $m$~lattice translates (an {\em $m$-periodic set}) 
in $\R^d$ is of the form 
\begin{equation}  \label{eqn:periodic-set-2}
\Lambda' = \bigcup_{i=1}^{m} \left( \vec{t}'_i + L \right)
,
\end{equation}
with a lattice $L\subset \R^d$
and translation vectors $\vec{t}'_i\in \R^d$, $i=1,\ldots ,m$.

We want to work with a parameter space for $m$-periodic sets similar to $\sdo$ for
lattices. For this, we consider $\Lambda'$ as 
a linear image $\Lambda'=A\Lambda_{\vec{t}}$ of a {\em standard periodic set}
\begin{equation}  \label{eqn:standard-periodic-set-def}
\Lambda_{\vec{t}}
=
\bigcup_{i=1}^{m} \left( \vec{t}_i + \Z^d \right)
.
\end{equation}
Here, $A\in\gldr$ satisfies in particular $L=A\Z^d$.
Since we are only interested in properties of 
periodic sets up to isometries, 
we encode $\Lambda'$ by $Q=A^tA\in\sdo$, 
together with the $m$ translation vectors $\vec{t}_1,\ldots, \vec{t}_m$.
Since every property of periodic sets we deal with here
is invariant up to translations, we may assume without loss of
generality that $\vec{t}_m=\vec{0}$.
Thus we consider the parameter space
\begin{equation}  \label{eqn:def-sdmo}
\sdmo = \sdo \times \R^{d\times (m-1)}
\end{equation}
for $m$-periodic sets (up to isometries).
We hereby in particular generalize the space 
$\MS^{d,1}_{>0}=\MS^d_{>0}$ in a natural way.
We call the elements of $\sdmo$ {\em periodic forms}  
and denote them usually by $X=(Q,\vec{t})$,
where $Q\in\sdo$ and 
$$
\vec{t}=(\vec{t}_1,\ldots,\vec{t}_{m-1})\in \R^{d\times (m-1)}
$$ 
is a real valued matrix containing $m-1$ columns with vectors $\vec{t}_i\in\R^d$.
One should keep in mind: although we omit $\vec{t}_m=\vec{0}$,
we implicitly keep it as a translation vector.
Note that a periodic set $\Lambda'$ as in \eqref{eqn:periodic-set-2}
has many {\em representations}
by periodic forms. In particular, $m$ may vary and we have different choices for $A$.
A similar approach for periodic sets in dimension~$3$ has been considered in~\cite{pz-1998}.

The parameter space $\sdmo$ is contained in the space 
\begin{equation}  \label{eqn:def-sdm}
\sdm = \sd \times \R^{d\times (m-1)}.
\end{equation}
It can be turned into a Euclidean space  
with inner product $\langle \cdot , \cdot \rangle$, defined for $X=(Q,\vec{t})$ and $X'=(Q',\vec{t}')$ by
$$   \label{not:inner_product_sdm}
\langle X, X'\rangle = \langle Q, Q' \rangle + \sum_{i=1}^{m-1} \vec{t}_i^t \vec{t}'_i
. 
$$
Note, for the sake of simplicity 
we use the same symbol for the inner products on all spaces $\sdm$.

We extend the definition of the arithmetical minimum $\lambda$, 
by defining the {\em generalized arithmetical minimum} 
$$   \label{not:arith_min_periodic_sets}
\lambda(X) 
=
\min \{ 
Q[\vec{t}_i-\vec{t}_j-\vec{v}] : 
1\leq i,j \leq m \mbox{ and } \vec{v}\in \Z^d, \mbox{ with } \vec{v}\not=\vec{0} \mbox{ if } i=j \}
$$
for the periodic form~$X=(Q,\vec{t})\in\sdmo$.
Note that we have $\lambda(X)=0$ in the case of 
intersecting lattice translates 
$(\vec{t}_i+\Z^d) \cap (\vec{t}_j+\Z^d) \not= \emptyset$
with $i\not=j$.
The set of {\em representations of the generalized arithmetical minimum} 
$\Min X$ is the set of all $\vec{w}=\vec{t}_i-\vec{t}_j-\vec{v}$    \label{not:Min-X}
attaining $\lambda(X)$.
Computationally, $\Min X$ and $\lambda(X)$ can be obtained by solving
a sequence of {\em closest vector problems}  
(CVPs), one for each pair $i,j$ with $i\not= j$. 
In addition one shortest vector problem (SVP) has to be solved, 
taking care of the cases where $i=j$. 
Implementations of algorithms solving CVPs and SVPs are provided for example
in {\tt MAGMA} \cite{magma} or {\tt GAP} \cite{gap}.

In order to define the sphere packing density function
$\delta:\sdmo \to \R$
we set $\det X = \det Q$ 
for periodic forms~$X=(Q,\vec{t})$. 
Then 
\begin{equation}  \label{eqn:periodic-delta}
\delta(X)
=
\left(\frac{\lambda (X)}{(\det X)^{1/d}}\right)^{\frac{d}{2}} m \vol B^d / 2^d
.
\end{equation}
In analogy to the lattice case, we 
call a periodic form~$X\in\sdmo$ \ {\em $m$-extreme}
if it attains a local maximum of~$\delta$ within~$\sdmo$.

The relation~\eqref{eqn:periodic-delta} shows that
the supremum of~$\delta$ among $m$-periodic sphere packings 
is up to some power and a constant factor equal to 
the ``Hermite like constant''
$$
\sup_{X\in\sdmo} \lambda(X) / (\det X)^{1/d} 
=
1 / \inf_{X\in\MNR_{m}} (\det X)^{1/d}
,
$$
where the set $\MNR_{m}$ on the right side
is the {\em (generalized) Ryshkov set}
\begin{equation}   \label{eqn:def-Pmlambda}
\MNR_{m}
=
\left\{
X\in\sdmo : \lambda(X)\geq 1
\right\}
.
\end{equation}

The condition $\lambda(X)\geq 1$ gives
infinitely many linear inequalities
$$
p_{\vec{v}}(X) = Q[\vec{v}] = \langle X , (\vec{v}\vec{v}^t,0) \rangle \geq 1
$$
for $\vec{v}\in\Z^d\setminus\{\vec{0}\}$,  
as in the case $m=1$. For $m>1$ we additionally have the 
infinitely many polynomial inequalities 
\begin{equation} \label{eqn:def-pijv}
p_{i,j,\vec{v}}(X)
=
Q[\vec{t}_i-\vec{t}_j-\vec{v}] 
\geq 1
,
\end{equation}
where $i,j\in\{1,\ldots, m\}$ with $i\not=j$ and $\vec{v}\in\Z^d$. 
These polynomials are
of degree~$3$ in the parameters $q_{kl}$, $t_{kl}$ of~$X$.
Note that they are linear for a fixed~$\vec{t}$.
Observe also that $p_{i,m,\vec{v}}$ and $p_{m,j,\vec{v}}$ are special 
due to our assumption $\vec{t}_m=\vec{0}$
and that there is a symmetry
$p_{i,j,\vec{v}}=p_{j,i,-\vec{v}}$ by which we may
restrict our attention to polynomials with $i\leq j$.
For $i=j$ we have the linear function $p_{i,j,\vec{v}}=p_{\vec{v}}$.

\section{Local analysis of periodic sphere packings}

\label{sec:local-analysis}

\subsection*{Characterizing local optima}

Before we generalize perfection and eutaxy to a notion of 
{\em $m$-perfection} 
and
{\em $m$-eutaxy}  
(in order to obtain
a sufficient condition for a periodic form to be $m$-extreme from it)
we discuss a rather general setting:
Assume we want to minimize a smooth function
on a {\em basic closed semialgebraic set},  
that is, on a region which is described
by finitely many (non-strict) polynomial inequalities.
Let $E$ denote a Euclidean space with inner product 
$\langle \cdot , \cdot \rangle$. Further, 
let $f:E\to \R$ be {\em smooth} 
(infinitely differentiable) 
and $g_1,\ldots, g_k$ be (real valued) polynomials on~$E$.
Assume we want to determine whether or not we have a local minimum
of $f$ at $X_0$ on the boundary of 
\begin{equation} \label{eqn:setG-definition}
G=\{ X \in E : g_i(X) \geq 0 \mbox{ for } i=1,\ldots,k \}
.
\end{equation}

For simplicity, we further assume $(\grad f)(X_0) \not= 0$ 
and $g_i(X_0)=0$, as well as $(\grad g_i)(X_0)\not= 0$, for $i=1,\dots, k$.
Then, in a sufficiently small neighborhood of $X_0$,
the function $f$ as well as the polynomials $g_i$ can be approximated arbitrarily 
close by corresponding affine functions.
For example, $f$ is approximated by the {\em beginning of its Taylor series} 
$$
f(X_0) + \langle (\grad f)(X_0) , X-X_0 \rangle  
.
$$
From this one easily derives the following well-known 
criterion (see for example \cite[Theorem~4.2.2]{bss-1993}) 
for an isolated local minimum of $f$ at $X_0$,
depending on the normal cone 
$$\MV(X_0)=\cone \{ (\grad g_i)(X_0) : i=1,\ldots,k \}.$$
The function $f$ attains an isolated local minimum on $G$ if
\begin{equation}  \label{eqn:grad-criterion}
(\grad f)(X_0) \in \interior \MV(X_0)
,
\end{equation}
and $f$ does not attain a local minimum if
\begin{equation}  \label{eqn:grad-criterion2}
(\grad f)(X_0) \not\in \MV(X_0)
.
\end{equation}
The behavior in the case $(\grad f)(X_0)\in \bd \cone \MV(X_0)$
depends on the involved functions $f$ and $g_i$ and has to be treated 
depending on the specific problem.

For the lattice sphere packing problem
we have $E=\sd$ and $f=\det^{1/d}$. For $Q_0\in\sdo$ 
we set $g_i(Q)=Q[\vec{v}_i]-\lambda(Q_0)$ with 
$(\grad g_i)(Q)= \vec{v}_i\vec{v}_i^t$
for each pair $\pm \vec{v}_i$ 
in $\Min Q_0$.
By Theorem \ref{thm:voronoi-packing} we have
a local minimum of $f(Q)=(\det Q)^{1/d}$ at $Q_0$
on $G$ (as in \eqref{eqn:setG-definition})
if and only if $Q_0$ is perfect and eutactic,
respectively if $\MV(Q_0)$ is full-dimensional 
and $(\grad f)(Q_0)\in \interior \MV(Q_0)$.
Here, $(\grad f)(Q_0)$ is a positive multiple of~$Q_0^{-1}$.
Thus in this special case (due to Proposition~\ref{prop:concave-det})
we do not have a local minimum of~$f$ where
$(\grad f)(Q_0)\in \bd \cone \MV(Q_0)$.

Let us consider the case of $m$-periodic sets, hence of $E=\sdm$ with $m>1$.
We want to know if a periodic form~$X_0\in\sdmo$ attains a local  
minimum of $f=\det^{1/d}$.
We may assume $\lambda(X_0) > 0$. The set $\Min X_0$ is finite
and moreover, for $X=(Q,\vec{t})$ in a small neighborhood of 
$X_0=(Q_0,\vec{t}^0)$, every $\vec{t}_i-\vec{t}_j-\vec{v}\in\Min X$
corresponds to a $\vec{t}^0_i-\vec{t}^0_j-\vec{v}\in\Min X_0$.
Thus locally at $X_0$, the generalized Ryshkov set $\MNR_{m}$
is given by the basic closed semialgebraic set~$G$ 
defined by the inequalities $p_{i,j,\vec{v}}(X)-\lambda(X_0)\geq 0$,  
one for each pair $\pm (\vec{t}^0_i-\vec{t}^0_j-\vec{v})$ 
in $\Min X_0$. 
As explained in Section~\ref{sec:periodic-parameter-space}, 
we may assume $1\leq i\leq j\leq m$ and $\vec{t}^0_j=\vec{0}$ if $j=m$.
An elementary calculation yields 
\begin{equation}   \label{eqn:pijv-gradient}
(\grad p_{i,j,\vec{v}})(X) = 
(
\vec{w}\vec{w}^t,
\vec{0},\dots,\vec{0},
2Q\vec{w},
\vec{0},\dots,\vec{0},
-2Q\vec{w},
\vec{0},\dots,\vec{0}
),
\end{equation}
where we set $X=(Q,\vec{t})$ and use
$\vec{w}$ to abbreviate $\vec{t}_i-\vec{t}_j-\vec{v}$.
This is to be understood as a vector in $\sdm=\sd\times \R^{d\times (m-1)}$,
with its ``$\sd$-component'' being the rank-$1$ form $\vec{w}\vec{w}^t$
and its ``translational-component'' containing the zero-vector~$\vec{0}$ 
in all but the $i$th and $j$th column.
If $j=m$, the $j$th column is omitted and if $i=j$
the corresponding column is $\vec{0}$.
For $(\grad f)(X)$ we obtain a positive multiple of $(Q^{-1},\vec{0})$.

\subsection*{A sufficient condition for local \texorpdfstring{$\mathbf{m}$}{m}-periodic sphere packing optima}

Generalizing the notion of perfection,
we say a periodic form~$X=(Q,\vec{t})\in \sdmo$ 
(and a corresponding periodic set represented by $X$)
is {\em $m$-perfect} if the {\em generalized Voronoi domain}
\begin{equation}  \label{eqn:periodic-voronoi-domain}
\MV(X)
= \cone
\{(\grad p_{i,j,\vec{v}})(X) : \vec{t}_i-\vec{t}_j-\vec{v}\in \Min X 
    \mbox{ for some } \vec{v}\in\Z^d \}
\end{equation}
is full-dimensional, that is, if $\dim \MV(X)= \dim \sdm = \binom{d+1}{2} + (m-1)d$. 
Generalizing the notion of eutaxy,
we say that $X$ (and a corresponding periodic set)
is {\em $m$-eutactic} if 
$$
(Q^{-1},\vec{0})\in\relint \MV(X).
$$
So the general discussion at the beginning of this section
yields the following sufficient condition
for a periodic form~$X$ to be {\em isolated $m$-extreme}, 
that is, for $X$ having the property that any sufficiently small change 
which preserves~$\lambda(X)$, necessarily lowers~$\delta(X)$.

\begin{theorem}  \label{thm:m-extreme-characterization}
If a periodic form~$X\in \sdmo$ is $m$-perfect and $m$-eutactic, then
$X$ is isolated $m$-extreme. 
\end{theorem}

Note that the theorem gives a computational tool to certify
isolated $m$-extremeness of a given periodic form~$X=(Q,\vec{t})\in \sdmo$:
First, we compute $\Min X$
and use equation~\eqref{eqn:pijv-gradient} 
to obtain generators of the generalized Voronoi domain $\MV(X)$.
From the generators it can be easily checked if the domain is full-dimensional, hence if $X$ is $m$-perfect.
Next, we can computationally test whether
$(Q^{-1},\vec{0})$ is in $\MV(X)$ or not; 
for example by solving a linear program similar to~\eqref{eqn:eutaxy-lp}.

If we find $(Q^{-1},\vec{0}) \in \relint \MV(X)$ (or equivalently in $\interior \MV(X)$ as $\MV(X)$ is assumed to be full-dimensional),
the periodic form~$X$ represents an isolated $m$-extreme periodic set.
If $(Q^{-1},\vec{0})\not\in\MV(X)$, 
the periodic form~$X$ does not represent an $m$-extreme periodic set.
In this situation, we can even 
find a ``direction'' $N\in\sdm$, for which we can improve 
the sphere packing density of the periodic form~$X$,
that is, such that $\delta(X+\epsilon N)>\delta(X)$ for
all sufficiently small $\epsilon>0$.

\begin{remark}
Let $X\in\sdmo$ with $(Q^{-1},\vec{0})\not\in\MV(X)$.
Then we can improve the sphere packing density of~$X$
in direction~$N$ given by the nearest point to 
$-(Q^{-1},\vec{0})$ in the polyhedral cone
\begin{equation}  \label{eqn:linear-cone-approx}
\MP(X)=
\{ N \in \sdm : \langle V, N \rangle \geq 0 \, \mbox{for all} \, V\in \MV(X) \}
.
\end{equation}
\end{remark}

Note that the cone~$\MP(X)$ is dual to the generalized Voronoi domain~$\MV(X)$
and (added to~$X$) gives locally a linear approximation of
the generalized Ryshkov set~$\MNR_{m}$. 

%

\subsection*{Fluid diamond packings}

For general $m$ we are confronted with a difficulty 
which does not show up in the lattice case $m=1$:
There may be non-isolated $m$-extreme sets, which are not $m$-perfect.
The {\em fluid diamond packings} 
in dimension $9$, described by
Conway and Sloane in \cite{cs-1995}, 
give such an example.

\begin{example} 
The {\em root lattice $\mathsf{D}_d$}  
can be defined by 
$$
\mathsf{D}_d = \{ \vec{x}\in \Z^d : \sum_{i=1}^d x_i \equiv 0 \!\! \mod 2 \}
.
$$
The {\em fluid diamond packings} are $2$-periodic sets
$$
\mathsf{D}_9\langle \vec{t} \rangle = \mathsf{D}_9 \cup (\mathsf{D}_9+\vec{t}) 
$$
with $\vec{t}\in\R^9$ such that the minimal distance among elements is 
equal to the minimum distance $\sqrt{2}$ of $\mathsf{D}_9$ itself.
We may choose for example~$\vec{t} = \vec{t}_{\alpha}=(\tfrac{1}{2},\ldots,\tfrac{1}{2},\alpha)^t$
with any $\alpha \in \R$.
For integers~$\alpha$ we obtain
the densest known packing lattice $\Lambda_9=\mathsf{D}_9\langle \vec{t}_{\alpha} \rangle$
in dimension $9$, showing that it is part of a family of 
uncountably many equally dense $2$-periodic sets.

The sets $\mathsf{D}_9\langle \vec{t}_{\alpha} \rangle$ give examples 
of non-isolated $2$-extreme sets, which are
$2$-eutactic, but not $2$-perfect. 
In order to see this, let us consider a representation 
$X_{\alpha}\in\MS^{9,2}_{>0}$ for $\mathsf{D}_9\langle \vec{t}_{\alpha} \rangle$.
We choose a basis $A$ of $\mathsf{D}_9$. 
Then $X_\alpha = (Q, A^{-1} \vec{t}_{\alpha})$,
with $Q=A^t A$,
is a representation of $\mathsf{D}_9\langle \vec{t}_{\alpha} \rangle$.

For non-integral $\alpha$ we find $\Min X_{\alpha}=\Min Q$ (using {\tt MAGMA} for example).
It follows (for example by Lemma~\ref{lem:strong-eutaxy} below) that
$X_\alpha$ is $2$-eutactic, but not $2$-perfect.
For integral $\alpha$ we find
$$
\Min X_{\alpha} = \Min Q \cup \{(x_1,\ldots,x_8,0)^t \in\{0,1\}^9 : \sum_{i=1}^8 x_i \equiv 0 \!\! \mod 2 \}
.
$$
Thus the vectors in $\Min X_{\alpha} \setminus \Min Q$ span only
an $8$-dimensional space.
Therefore $X_{\alpha}$ is not $2$-perfect. 
Nevertheless, a corresponding calculation shows that $X_{\alpha}$ is $2$-eutactic, 
as in the case of non-integral $\alpha$.

In order to see that $X_{\alpha}$ is non-isolated $2$-extreme,
we can apply Proposition~\ref{prop:sufficient-for-m-extreme} below.
One easily checks that for integral~$\alpha$
(hence for the lattice $\Lambda_9$) we have only one
degree of freedom for a local change of~$\vec{t}_{\alpha}$ 
giving an equally dense sphere packing. 
For non-integral~$\alpha$ we have nine degrees of freedom 
for such a modification.
\newline\rightline{$\Box$}
\end{example}

Non-isolated $m$-extreme sets as in this
example can occur for periodic forms~$X\in\sdmo$,
only if $(Q^{-1},\vec{0})\in \bd \MV(X)$
(which is for example always the case if $X$ is $m$-eutactic, but not $m$-perfect). 
In this case it is in general not clear what
an infinitesimal change of $X$ in a direction $N\in\sdm$ leads to
(already assuming it is orthogonal to $(Q^{-1},\vec{0})$ as well as
in the boundary of the set $\MP(X)$ in~\eqref{eqn:linear-cone-approx}).
If $\MF(X)$ denotes the unique face of $\MV(X)$
containing $(Q^{-1},\vec{0})$ in its relative interior,
then this ``set of uncertainty'' is equal to
the face of $\MP(X)$ dual to $\MF(X)$, that is, equal to
\begin{equation} \label{eqn:set-of-uncertainty}
\MU(X)=
\{ N \in \MP(X) : \langle V, N \rangle = 0  \; \mbox{for all} \, V\in \MF(X) 
\}
.
\end{equation}
Or in other words, the set $\MU(X)$
is the intersection of $\MP(X)$ with the hyperplane
orthogonal to $(Q^{-1},\vec{0})$.
Note that it is possible to determine $\MF(X)$ 
(and hence a description of $\MU(X)$ by linear inequalities)
computationally, using linear programming techniques.

\subsection*{Purely translational changes}

Below we give an additional sufficient condition for $m$-extremeness.
For this we consider the case when
all directions in $\MU(X)$ are 
``{\em purely translational changes}'' $N = (0,\vec{t}^N) \in \sdm$.
A vivid interpretation of
a purely translational change 
can be given by thinking of the corresponding modification 
of a periodic sphere packing. The spheres of each lattice translate
are jointly moved. If in such a local change all contacts 
among spheres are lost, we can increase their radius and
obtain a new sphere packing with larger density. 
If some contacts among spheres are preserved however, 
the sphere packing density remains the same.
The latter case is captured in the following proposition,
which gives an easily testable criterion 
for $m$-extremeness. 
We apply this proposition in Section~\ref{sec:periodic-extreme},
where we consider potential local improvements of
best known packing lattices to periodic non-lattice sets.

\begin{proposition}    \label{prop:sufficient-for-m-extreme}
For a periodic form~$X=(Q,\vec{t})\in \sdmo$ 
with $(Q^{-1},\vec{0})\in \bd \MV(X)$,
let $\MU(X)$ be contained in 
$$
\{ 
(\vec{0},\vec{t}^N)\in\sdm : \vec{t}^N_i=\vec{t}^N_j \; 
\mbox{for at least one} \, \vec{t}_i-\vec{t}_j-\vec{v}\in\Min X \;\mbox{with}\; \vec{v}\in\Z^d
\}
.
$$
Then $X$ is (possibly non-isolated) $m$-extreme.
\end{proposition}

Note, if $X$ is $m$-eutactic (possibly not $m$-perfect), 
the set $\MU(X)$ is the {\em orthogonal complement} 
$\MV(X)^\perp$ of the linear hull of $\MV(X)$. 
Note also that Proposition~\ref{prop:sufficient-for-m-extreme} 
includes in particular the special case
where some $\vec{v}\in\Z^d$ are in $\Min X$
(and therefore $\vec{t}_i=\vec{t}_j=\vec{0}$ for $i=j=m$). 
This situation occurs for the $2$-periodic, fluid diamond packings
in the example above.

From the sphere packing interpretation of the proposition its assertion
is clear. Nevertheless, we give a proof below, based on
a local analysis in $\sdmo$. More than actually needed for the proof,
we analyze how $\delta$ changes locally
at a periodic form~$X\in\sdmo$ in a direction $N\in\MU(X)$.
As a byproduct, we obtain tools allowing a
computational analysis of possible local optimality 
for a given periodic form (not necessarily covered by the proposition).
These can for example be used in a numerical
search for good periodic sphere packings.

\begin{proof}[Proof of Proposition \ref{prop:sufficient-for-m-extreme}]
The generalized Voronoi domain $\MV(X)$ is spanned by 
gradients $(\grad p_{i,j,\vec{v}})(X)$ 
(as given in \eqref{eqn:pijv-gradient}),
one for each pair of vectors $\pm\vec{w}\in\Min X$.
The assumption that a direction $N=(Q^N,\vec{t}^N)$
is in $\MU(X)$ for a periodic form~$X=(Q,\vec{t})$, 
implies $\langle Q^{-1}, Q^N \rangle = 0$.
Moreover, for the unique maximal face $\MF(X)$ of $\MV(X)$ with $(Q^{-1},\vec{0})\in \relint \MF(X)$,
the condition that $N$ is orthogonal to some 
$(\grad p_{i,j,\vec{v}})(X)$ in $\MF(X)$  
translates into 
\begin{equation}  \label{eqn:gradient-condition}
\langle (\grad p_{i,j,\vec{v}})(X), N \rangle
= Q^N[\vec{w}] + 2 (\vec{t}_i^N-\vec{t}_j^N)^t Q \vec{w}
= 0
,
\end{equation}
with $\vec{w}=(\vec{t}_i-\vec{t}_j-\vec{v})$.
Recall that in the special case $i=j$ (and for $m=1$ anyway) 
$p_{i,j,\vec{v}}$ is linear and \eqref{eqn:gradient-condition} reduces to the 
condition $Q^N[\vec{w}]=0$; if then $N$ satisfies this linear condition, 
$p_{i,j,\vec{v}}(X+\epsilon N)$ is a constant function in $\epsilon$.

When $p_{i,j,\vec{v}}(X+\epsilon N)$ is a cubic polynomial 
in $\epsilon$ we need to use higher order information
in order to judge its behavior. 
An elementary calculation yields for the {\em Hessian}
\begin{equation}  \label{eqn:hessian-evaluation}
(\hess p_{i,j,\vec{v}})(X)[N] 
= 
2 Q [\vec{t}_i^N-\vec{t}_j^N] + 4 (\vec{t}_i^N-\vec{t}_j^N)^t Q^N \vec{w}
.
\end{equation}

Now how does $\delta$ change at~$X$ in direction~$N$, 
assuming it is in the set of uncertainty $\MU(X)$?
Among the polynomials $p_{i,j,\vec{v}}$ 
with $N$ satisfying \eqref{eqn:gradient-condition},
the fastest decreasing polynomial in direction $N$ determines 
$\lambda(X+\epsilon  N)$ for small enough $\epsilon$. 
Thus for the local change of $\delta$ in direction $N$,
we may restrict our attention to a polynomial $p_{i,j,\vec{v}}$ 
with the smallest value \eqref{eqn:hessian-evaluation} of its Hessian.

By Proposition~\ref{prop:concave-det} we know that
$\det^{1/d}$ decreases strictly at $X$
in a direction~$N\in\MU(X)$ 
if and only if $Q^N\not=0$.

For a purely translational change with $Q^N=0$, 
the function $\det^{1/d}$ remains constant.
On the other hand, because of \eqref{eqn:hessian-evaluation} and 
since $Q$ is positive definite,
we have $(\hess p_{i,j,\vec{v}})(X)[N]\geq 0$, 
with equality if and only if $\vec{t}_i^N-\vec{t}_j^N=\vec{0}$.
The latter implies
that $p_{i,j,\vec{v}}(X+\epsilon N)$ is a constant function of 
$\epsilon$.
Thus for purely translational changes $N=(0,\vec{t}^N)\in \MU(X)$,
the density function $\delta(X+\epsilon N)$ is constant
for small enough $\epsilon \geq 0$, 
if $\vec{t}_i^N=\vec{t}_j^N$ for
some pair $(i,j)$ with 
$\vec{t}_i-\vec{t}_j-\vec{v}\in\Min X$ 
(for a suitable $\vec{v}\in\Z^d$).
This proves the proposition.
\end{proof}

Note that our argumentation in the proof 
also shows that $\delta(X+\epsilon N)$
increases for small $\epsilon > 0$,
for a purely translational change $N=(0,\vec{t}^N)\in \MU(X)$
with $\vec{t}_i^N\not=\vec{t}_j^N$ for all
pairs $(i,j)$ with $\vec{t}_i-\vec{t}_j-\vec{v}\in\Min X$
(for some $\vec{v}\in\Z^d$).
This case corresponds to a modification of a periodic sphere packing
in which all contacts among spheres are lost.

\section{Periodic extreme sets}

\label{sec:periodic-extreme}

A given periodic set has many representations by periodic forms,
in spaces $\sdmo$ with varying $m$.
For example, by choosing some sublattice of $\Z^d$, we
can add additional translational parts.

It could happen that a periodic set $\Lambda$
with a given representation $X\in\sdmo$ is
$m$-extreme, whereas a second representation $X'\in\MS^{d,m'}$
is not $m'$-extreme. We are not aware of an example though. 
However, in some cases we are certain that the packing density of 
no representation of~$\Lambda$ can locally be improved.

\begin{definition} \label{def:periodic-extreme}
A periodic set is {\em periodic extreme} 
if it is $m$-extreme for
all possible representations $X\in\sdmo$.
\end{definition}

Theorem \ref{thm:main-periodic} below gives a sufficient condition for a lattice to be periodic extreme.
For its statement we need the notion of {\em strong eutaxy}   
for lattices, respectively PQFs:
A form $Q\in\sdo$ (and a corresponding lattice) is called {\em strongly eutactic} 
if 
\begin{equation}  \label{eqn:strong-eutaxy-condition}
Q^{-1} = \alpha \sum_{\vec{x}\in\Min Q} \vec{x}\vec{x}^t
\end{equation}
for some $\alpha>0$, i.e., if the coefficients in the eutaxy 
condition \eqref{eqn:eutaxy-algebraic} are all equal.
It is well-known that a PQF~$Q$ is strongly eutactic if and only if the 
vectors in $\Min Q$ form a so-called {\em spherical $2$-design}  
with respect to $Q$ (see \cite{venkov-2001}, \cite[Corollary~16.1.3]{martinet-2003}).

\begin{lemma} \label{lem:strong-eutaxy}
Any representation $X\in\sdmo$ of a strongly eutactic lattice (respectively PQF) is $m$-eutactic.
\end{lemma}

\begin{proof}
Let $Q\in\sdo$ be strongly eutactic,
satisfying \eqref{eqn:strong-eutaxy-condition} for some $\alpha>0$.
Let $X=(Q^X,\vec{t}^X)\in\sdmo$ be some representation of~$Q$, 
e.g. with $m>1$.
Let the corresponding eutactic lattice be denoted by~$\Lambda$.
Then $Q^X$ is the Gram matrix of a basis $A\in \gldr$ 
of a sublattice~$L$ of~$\Lambda$. The columns of $\vec{t}^X$ 
are the coordinates of lattice points of~$\Lambda$ relative to~$A$.

For a fixed $\vec{w}\in \Min X$ we define an abstract graph, 
whose vertices are the indices in $\{1,\ldots,m\}$.
Two vertices $i$ and $j$ are connected by an edge 
whenever there is some $\vec{v}\in \Z^d$ such that $\vec{w}=\vec{t}^X_i-\vec{t}^X_j-\vec{v}$.
In other words, the graph reflects via an edge $(i,j)$ that spheres
of packing radius $\lambda(\Lambda)$ around points of the two sublattice translates
$A(\vec{t}^X_i+Z^d)$ and $A(\vec{t}^X_j+Z^d)$ touch.
For $\vec{z}\in Z^d$, the sphere with center~$A(\vec{t}^X_j+\vec{z})$ touches the sphere with center~$A(\vec{t}^X_j+\vec{z}+\vec{w})$.
Since the periodic form~$X$ represents a lattice $\Lambda$, we find a
chain of touching spheres at centers~$A(\vec{t}^X_j+\vec{z}+k\vec{w})$, with $k=0,1,\ldots$.
Modulo some natural number less or equal to~$m$ these centers belong to the 
same lattice translate of~$L$.
As a consequence, we find that the graph defined above is a disjoint union of cycles. 
So $\vec{w}$ induces a partition $(I_1,\dots,I_k)$ of $\{1,\ldots,m\}$.

Let~$I$ be an index set of this partition 
(containing the indices of a fixed cycle of the defined graph).
Summing over all triples $(i,j,\vec{v})$ with $i,j\in I$ and $\vec{v}\in\Z^d$ 
such that $\vec{w}=\vec{t}^X_i-\vec{t}^X_j-\vec{v}\in\Min X$,
we find (using \eqref{eqn:pijv-gradient}):
$$
\sum_{ \genfrac{}{}{0pt}{2}{(i,j,\vec{v}) \in I^2\times\Z^d}{\mbox{\tiny with} \; \vec{v}=\vec{t}^X_i-\vec{t}^X_j-\vec{w}}} 
(\grad p_{i,j,\vec{v}})(X)
=
2|I| (\vec{w}\vec{w}^t, \vec{0})
.
$$
The factor $2$ comes from the symmetry 
$\grad p_{i,j,\vec{v}}=\grad p_{j,i,-\vec{v}}$.
Summation over all index sets $I$ of the partition yields
\begin{equation}   \label{eqn:summing_gradients}
\sum_{\genfrac{}{}{0pt}{2}{(i,j,\vec{v}) \in \{1,\ldots,m\}^2\times\Z^d}{\mbox{\tiny with} \; \vec{v}=\vec{t}^X_i-\vec{t}^X_j-\vec{w}}} 
(\grad p_{i,j,\vec{v}})(X)
= 2 m (\vec{w}\vec{w}^t, \vec{0})
.
\end{equation}
As a consequence we find by the strong eutaxy condition~\eqref{eqn:strong-eutaxy-condition} that 
$$
(Q^{-1},\vec{0})
\; = \;
(\alpha/2m)
\sum_{\genfrac{}{}{0pt}{2}{\vec{w}\in\Min X,   (i,j,\vec{v}) \in \{1,\ldots,m\}^2\times\Z^d}{\mbox{\tiny with} \; \vec{v}=\vec{t}^X_i-\vec{t}^X_j-\vec{w} }}
(\grad p_{i,j,\vec{v}})(X)
,
$$
with a suitable $\alpha>0$.
Thus $X$ is $m$-eutactic.
\end{proof}

Not all PQFs (or lattices) which are 
strongly eutactic have to be perfect.
For example the lattices~$\Z^n$ for $n \geq 2$ are of this kind.
But if a strongly eutactic PQF is in addition also perfect, 
then the following theorem shows that this 
is sufficient for it to be periodic extreme.
Note that this applies in particular to so called 
{\em strongly perfect} lattices and PQFs. 
For these lattices the vectors in $\Min Q$ form a spherical $4$-design 
with respect to~$Q$ (see \cite{nebe-2002} or \cite[Chapter~16]{martinet-2003} for further details).

\begin{theorem} \label{thm:main-periodic}
Perfect, strongly eutactic lattices (respectively PQFs) are periodic extreme.
\end{theorem}

\begin{proof}
Let $Q\in\sdo$ be perfect and strongly eutactic. 
Hence the vectors in $\Min Q$ span $\R^d$ (by Proposition \ref{prop:prefect-implies-d-linear-indep})
and satisfy \eqref{eqn:strong-eutaxy-condition} for some $\alpha>0$.
Let $X=(Q^X,\vec{t}^X)\in\sdmo$ be a representation of $Q$.
By Lemma~\ref{lem:strong-eutaxy}, $X$ is $m$-eutactic.
If $X$ is $m$-perfect as well, we know by Theorem \ref{thm:m-extreme-characterization}
that $X$ is also $m$-extreme. 

So let us assume that $X$ is not $m$-perfect; hence the generalized Voronoi domain~$\MV(X)$ 
is not full-dimensional. We want to apply Proposition~\ref{prop:sufficient-for-m-extreme}.
For this we choose 
$$
N=(Q^N,\vec{t}^N)\in \MU(X)=\MV(X)^{\perp} \quad \mbox{with} \quad N\not=0
.
$$
(Recall the definition of $\MU(X)$ from~\eqref{eqn:set-of-uncertainty}
and that $\MU(X)=\MV(X)^{\perp}$ if $X$ is $m$-eutactic.)
By this assumption we have in particular
$$\langle N, (\grad p_{i,j,\vec{v}})(X) \rangle = 0$$
for all triples $(i,j,\vec{v})$ with $\vec{w}=\vec{t}^X_i-\vec{t}^X_j-\vec{v}\in \Min X$.
Using equation~\eqref{eqn:summing_gradients}, which we obtained in the proof of Lemma~\ref{lem:strong-eutaxy},
we get $\langle N, (\vec{w}\vec{w}^t,\vec{0}) \rangle = Q^N[\vec{w}]=0$ 
for every fixed $\vec{w}\in \Min X$.

By Proposition~\ref{prop:prefect-implies-d-linear-indep}
there exist $d$~linearly independent $\vec{w}$ in $\Min X$,
which implies $Q^N=0$.
Using \eqref{eqn:gradient-condition}, we obtain
\begin{equation}   \label{eqn:orthogonality-reduced-to}
0 = \langle N, (\grad p_{i,j,\vec{v}})(X) \rangle = 2 (\vec{t}_i^N-\vec{t}_j^N)^t Q \vec{w}
.
\end{equation}
If $\vec{t}_i^N-\vec{t}_j^N=\vec{0}$ for some pair $(i,j)$
we can apply Proposition \ref{prop:sufficient-for-m-extreme}.
Note that this includes in particular the case $i=j=m$
($\vec{t}_i^N=\vec{t}_j^N=\vec{0}$) if $\vec{v}\in\Z^d\cap \Min X$.
So we may assume that such $\vec{v}$ do not exist.

We choose~$d$~linearly independent vectors $\vec{w}_1,\ldots , \vec{w}_d \in \Min X$
(that exist by Proposition~\ref{prop:prefect-implies-d-linear-indep}).
By the assumption that non of the $\vec{w}_i$ is integral and by the assumption
that $X$ represents a lattice, each $\vec{w}_i$ connects the origin~$\vec{t}^X_m=\vec{0}$
to another translation vector $\vec{t}^X_j$ (with $j\not=m$) via $\vec{w}_i = \vec{t}^X_j -\vec{v}$ for some $\vec{v}\in \Z^d$. 
In the same way each of the chosen minimal vectors connects the translation vector with index~$i$ to other translation vectors.
We denote by~$I$ the subset of $\{1,\dots, m\}$ that is connected to
the index~$m$ (respectively to the origin~$\vec{0}$) via a sequence of 
such links through the chosen $d$~minimal vectors.
For each index~$i\in I$ we get from the minimal vectors 
$d$~independent linear conditions~\eqref{eqn:orthogonality-reduced-to}
for the differences $\vec{t}_i^N-\vec{t}_j^N$, with suitable $j\in I\setminus\{i\}$. 
Overall we obtain $d|I|$~independent equations for~$d|I|$ differences. 
We deduce that all of them vanish. Moreover, as $\vec{t}_m^N = \vec{0}$ we even find
$\vec{t}_i^N = \vec{0}$ for all indices~$i\in I$.

\end{proof}

The root lattices $\mathsf{A}_d$, $\mathsf{D}_d$ and $\mathsf{E}_d$, as well as the Leech lattice
are known to be perfect and strongly eutactic (cf.~\cite{martinet-2003}).
These lattices are known to solve the 
lattice sphere packing problem in dimensions $d\leq 8$ and $d=24$ 
(see Table~\ref{tab:sphere-packing-results}).
As an immediate consequence of Theorem~\ref{thm:main-periodic}, 
we find that they cannot locally be improved to a periodic non-lattice set
with greater sphere packing density.

\begin{corollary}  \label{cor:root_lattice_periodic_extreme}
The lattices $\mathsf{A}_d$, for $d\geq 2$, $\mathsf{D}_d$, for $d\geq 3$, 
and $\mathsf{E}_d$, for $d=6,7,8$, as well as the Leech lattice
are periodic extreme.
\end{corollary}

We also checked whether or not Theorem~\ref{thm:main-periodic} can be applied to 
other dimensions~$d \leq 24$. 
For these dimensions the so-called {\em laminated lattices}~$\Lambda_d$  
and {\em sections $K_d$ of the Leech lattice}  
give the densest known lattice sphere packings.
The lattices~$K_d$ are different from $\Lambda_d$
(and at the same time give the densest known lattice sphere packings) 
only in dimensions $d=11,12,13$.
For these $d$, the lattice $K_d$ is strongly eutactic only for $d=12$,
when $K_d$ is also known as {\em Coxeter-Todd lattice}.  
The laminated lattices $\Lambda_d$ give the densest known packing lattices
in dimensions $d=9,10$ and $d=14,\dots,24$ (for $d=18,\dots,24$ they coincide with $K_d$).
Among those values for $d$, the laminated lattices $\Lambda_d$ are strongly eutactic
if and only if $d=15,16$ or $d\geq 20$.
Concluding, we cannot exclude that
densest known lattice sphere packings in dimensions $d\in\{9,10,11,13,14,17,18,19\}$ can
locally be improved to better periodic sphere packings.
Further analysis is required here.

\section{Floating and strict periodic extreme lattices}

The last step of the proof of Theorem~\ref{thm:main-periodic} 
has a vivid interpretation if we 
think of a sphere packing described by the given lattice.
Let $X=(Q^X,\vec{t}^X)$ be one of its representations and let 
$A$ denote a sublattice basis with Gram matrix $Q^X$.
Then the sublattice translates $A(\vec{t}^X_i+Z^d)$ with $i\in I$ form a 
``rigid component'' of the sphere packing.
If we do not want to decrease the sphere packing density
in a local deformation we have to move all of its translates simultaneously.
This rigid component may actually be larger than the one used in the 
proof of Theorem~\ref{thm:main-periodic}. 
It may even consist of the whole packing.
A maximal rigid component of translates can 
be described via an abstract graph with vertices in $\{1,\ldots,m\}$:
$(i,j)$ is an edge
whenever there is some $\vec{v}\in \Z^d$ such that $\vec{t}^X_i-\vec{t}^X_j-\vec{v}\in\Min X$. 
Let $I$ be the set of indices~$i$ (vertices of the graph) connected by a path with~$m$.
If $|I|=m$ the whole packing forms one rigid component.
If $I$ is a strict subset of $\{1,\ldots,m\}$
we say a corresponding packing or lattice is {\em floating}.

In a floating packing each connected component of the graph 
above corresponds to a union of translates which can jointly locally be moved 
without changing~$\lambda (X)$ and~$\delta(X)$ respectively.
Examples are the fluid diamond packings described in the example of
Section~\ref{sec:local-analysis}.
The same applies to their higher-dimensional generalizations 
$\mathsf{D}_d^+ = \mathsf{D}_d \cup \left( \mathsf{D}_d + ( \tfrac{1}{2},\ldots, \tfrac{1}{2} ) \right)$
for $d\geq 10$ (see \cite[Section~4.7.3]{cs-1998}). 
For even~$d$ these $2$-periodic sets are actually lattices (hence $1$-periodic).
In fact $\mathsf{E}_8 = D_8^+$. 

\bigskip

We note that Theorem~\ref{thm:main-periodic} and  
Corollary~\ref{cor:root_lattice_periodic_extreme} give statements about local optimality of lattices, but not about strict
local optimality. With the assumptions of Theorem~\ref{thm:main-periodic}
alone strict local optimality cannot be ensured, as shown by floating lattices like $D_d^+$ for even~$d\geq 10$.
These lattices have the same minimal vectors as the corresponding root lattice~$D_d$
and therefore they give a series of perfect and strongly eutactic lattices that are periodic extreme by Theorem~\ref{thm:main-periodic}.
However, they can locally be modified to other $2$-periodic sets of the same density.

We think that a strengthening of Corollary~\ref{cor:root_lattice_periodic_extreme} is
possible for certain lattices that are non-floating, perfect and strongly eutactic.
These include in particular the $\mathsf{E}_8$ root lattice and the Leech lattice.
We think it is possible to show that these lattices are  {\em strict periodic extreme},
meaning they are isolated $m$-extreme for all possible representations $X\in\sdmo$.
By Lemma~\ref{lem:strong-eutaxy} and Theorem~\ref{thm:m-extreme-characterization} 
one has to show that a given non-floating, perfect and strongly eutactic lattice 
is $m$-perfect for every~$m$. Here some further work is required...

\section*{Acknowledgement}

The author thanks Henry Cohn, Renaud Coulangeon, Jacques Martinet, Frank Vallentin
and Giovanni Zanzotto for many useful discussions.
He moreover thanks an anonymous referee for several helpful suggestions.

%


\begin{thebibliography}{BBC98}

\bibitem[Bat01]{batut-2001}
C.~Batut, \emph{Classification of quintic eutactic forms}, Math. Comp.
  \textbf{70} (2001), 395--417.

\bibitem[BSS93]{bss-1993}
M.S.~Bazaraa, H.D.~Sherali, C.M.~Shetty, {\em Nonlinear
programming: Theory and algorithms}, Wiley, 1993.

\bibitem[BE00]{be-2000}
J.~Bierbrauer and Y.~Edel, \emph{{Dense sphere packings from new codes}}, J.
  Algebr. Comb. \textbf{11} (2000), 95--100.

\bibitem[Bli35]{blichfeldt-1934}
H.F. Blichfeldt, \emph{The minimum values of positive quadratic forms in six,
  seven and eight variables}, Math. Z. \textbf{39} (1935), 1--15.

\bibitem[BM96]{bm-1996}
A.M. Berg{\'e} and J.~Martinet, \emph{Sur la classification des r\'eseaux
  eutactiques}, J. London Math. Soc. (2) \textbf{53} (1996), 417--432.

\bibitem[CE03]{ce-2003}
H.~Cohn and N.~Elkies, \emph{New upper bounds on sphere packings I}, Ann. Math. {\bf 157} (2003), 689–-714.

\bibitem[CK09]{ck-2004}
H.~Cohn and A.~Kumar, \emph{Optimality and uniqueness of the {L}eech lattice
  among lattices}, Ann. Math. {\bf 170} (2009), 1003–-1050.

\bibitem[CS95]{cs-1995}
J.H. Conway and N.J.A. Sloane, \emph{What are all the best sphere packings in
  low dimensions?}, Discrete Comput. Geom. \textbf{13} (1995), 383--403.

\bibitem[CS96]{cs-1996}
\bysame, \emph{The antipode construction for sphere packings}, Invent. Math.
  \textbf{123} (1996), 309--313.

\bibitem[CS99]{cs-1998}
\bysame, \emph{Sphere packings, lattices and groups} (3rd ed.), Springer, 1999.

\bibitem[DSV07]{dsv-2007b}
M.~D{utour Sikiri\'c}, A.~Sch\"urmann, and F.~Vallentin, \emph{Classification
  of eight dimensional perfect forms}, Electron. Res. Announc. Amer. Math. Soc.
  \textbf{13} (2007), 21--32.

\bibitem[EGS02]{egs-2002}
P.~E{lbaz-Vincent}, H.~Gangl, and C.~Soul{\'e}, \emph{Quelques calculs de la
  cohomologie de {${\rm GL}\sb N(\mathbb Z)$} et de la {$K$}-th\'eorie de
  {$\mathbb Z$}}, C. R. Math. Acad. Sci. Paris \textbf{335} (2002), 321--324.

\bibitem[Gau40]{gauss-1840}
C.F. Gauss, \emph{Untersuchungen \"uber die {E}igenschaften der positiven
  tern\"aren quadratischen {F}ormen von {L}udwig {A}ugust {S}eeber}, J. Reine.
  Angew. Math. \textbf{20} (1840), 312--320.

\bibitem[Gro63]{groemer-1963}
H. Groemer, \emph{{E}xistenzs\"atze f\"ur {L}agerungen im {E}uklidischen {R}aum}, 
Math. Z. \textbf{81} (1963), 260–-278

\bibitem[GL87]{gl-1987}
P.M. Gruber and C.G. Lekkerkerker, \emph{Geometry of numbers}, North--Holland, 1987.

\bibitem[Gru07]{gruber-2007}
P.M. Gruber, \emph{Convex and discrete geometry}, Springer, 2007.

\bibitem[Hal05]{hales-2005}
T.C. Hales, \emph{{A proof of the {K}epler conjecture}}, Ann. Math \textbf{162}
  (2005), 1065--1185.

\bibitem[KZ77]{kz-1877}
A.~Korkine and G.~Zolotareff, \emph{Sur les formes quadratiques positives},
  Math. Ann. \textbf{11} (1877), 242--292.

\bibitem[Lag73]{lagrange-1773}
J.L. Lagrange, \emph{Recherches d'arithm\'etique}, Nouv. M\'em. Acad. Berlin
  (1773), 265--312, in Oeuvres de Lagrange III, 695--795.

\bibitem[LS70]{ls-1970}
J.~Leech and N.J.A.~Sloane, \emph{{New sphere packings in dimensions {$9$ -- $15$}}}, 
Bull. Am. Math. Soc. \textbf{76} (1970), 1006--1010.

\bibitem[Mar03]{martinet-2003}
J.~Martinet, \emph{Perfect lattices in {E}uclidean spaces}, Springer, 2003.

\bibitem[Min05]{minkowski-1905}
H.~Minkowski, \emph{Diskontinuit\"atsbereich f\"ur arithmetische
  \"{A}quivalenz}, J. Reine Angew. Math. \textbf{129} (1905), 220--274, Reprint
  in {\em Gesammelte Abhandlungen}, Band II, Teubner, 1911.

\bibitem[Neb02]{nebe-2002}
G.~Nebe, \emph{Gitter und {M}odulformen}, Jahresber. Deutsch. Math.-Verein.
  \textbf{104} (2002), no.~3, 125--144.

\bibitem[Nel74]{nelson-1974}
C.E.~Nelson, \emph{{The reduction of positive definite quinary quadratic
  forms}}, Aequationes Math. \textbf{11} (1974), 163--168.

\bibitem[PZ98]{pz-1998}
M.~Pitteri and G.~Zanzotto, \emph{Beyond Space Groups: the Arithmetic Symmetry of Deformable Muitilattices},
Acta Cryst. \textbf{54} (1998), 359--373.

\bibitem[Rys70]{ryshkov-1970}
S.S. Ryshkov, \emph{The polyhedron {$\mu (m)$} and certain extremal problems of
  the geometry of numbers}, Soviet Math. Dokl. \textbf{11} (1970), 1240--1244,
  translation from Dokl. Akad. Nauk SSSR 194, 514--517 (1970).

\bibitem[Sch09]{schuermann-2009}
A.~Sch\"urmann, \emph{Computational geometry of positive definite quadratic
  forms},  University Lecture Series {\bf 48}, AMS, 2009.

\bibitem[Sto75]{stogrin-1974}
M.I. Stogrin, \emph{Locally quasidensest lattice packings of spheres}, Soviet
  Math. Dokl. \textbf{15} (1975), 1288--1292, translation from Dokl. Akad. Nauk
  SSSR 218, 62--65 (1974).

\bibitem[Thu10]{thue-1910}
A.~Thue, \emph{{\"Uber die dichteste {Z}usammenstellung von kongruenten
  {K}reisen in einer {E}bene}}, Norske Vid. Selsk. Skr. \textbf{1} (1910),
  1--9.

\bibitem[Var95]{vardy-1995}
A.~Vardy, \emph{A new sphere packing in {$20$} dimensions}, Invent. Math.
  \textbf{121} (1995), 119--133.

\bibitem[Ven01]{venkov-2001}
B.B. Venkov, \emph{R\'eseaux et designs sph\'eriques}, R\'eseaux euclidiens,
  designs sph\'eriques et formes modulaires, Monogr. Enseign. Math., vol.~37,
  Enseignement Math., Geneva, 2001, pp.~10--86.

\bibitem[Vor07]{voronoi-1907}
G.F. Voronoi, \emph{Nouvelles applications des param\`etres continus \`a la
  th\'eorie des formes quadratiques. {P}remier {M}\'emoire. {S}ur quelques
  propri\'et\'es des formes quadratiques positives parfaites}, J. Reine Angew.
  Math. \textbf{133} (1907), 97--178.

\bibitem[Zie97]{ziegler-1998}
G.M. Ziegler, \emph{Lectures on polytopes}, Springer, 1997.


\bigskip


\centerline{\sc Software and Webpages}

\medskip

\bibitem[GAP]{gap}
\emph{The {GAP} group, {GAP} --- {G}roups, {A}lgorithms, {P}rogramming - a
  system for computational discrete algebra}, ver. 4.4,
  \url{http://www.gap-system.org/}.


\bibitem[MAG]{magma}
\emph{Computational {A}lgebra {G}roup, {U}niversity of {S}ydney, {MAGMA} ---
  high performance software for {A}lgebra, {N}umber {T}heory, and {G}eometry},
  ver. 2.13, \url{http://magma.maths.usyd.edu.au/}.


\bibitem[NS]{ns-2008}
G.~Nebe and N.J.A. Sloane, \emph{A catalogue of lattices}, published
  electronically at \url{http://www.research.att.com/~njas/lattices/}.


\end{thebibliography}

\providecommand{\bysame}{\leavevmode\hbox to3em{\hrulefill}\thinspace}
\providecommand{\MR}{\relax\ifhmode\unskip\space\fi MR }
\providecommand{\MRhref}[2]{%
  \href{http://www.ams.org/mathscinet-getitem?mr=#1}{#2}
}
\providecommand{\href}[2]{#2}

\end{document}